\documentclass[12pt, reqno]{amsart}
\usepackage{amsmath}
\usepackage{amssymb}
\usepackage{amsfonts}
\usepackage{amsthm}
\usepackage{relsize}
\usepackage{setspace}
\usepackage{geometry}
\usepackage{enumitem}
\usepackage{url}
\usepackage{xspace}
\usepackage{mathrsfs}
\usepackage{dsfont}
\usepackage{lmodern}
\usepackage{xcolor}
\usepackage[utf8]{inputenc}
\usepackage{mathtools}

\usepackage{hyperref}

\usepackage[textsize=footnotesize,textwidth=15ex]{todonotes}

\newtheorem{thm}{Theorem}[section]
\newtheorem{lem}[thm]{Lemma}
\newtheorem{prop}[thm]{Proposition}
\newtheorem{cor}[thm]{Corollary}

\theoremstyle{definition}
\newtheorem{remark}[thm]{Remark}

\numberwithin{equation}{section}

\newcommand{\FF}{\mathbf{F}}

\newcommand{\NN}{\mathbf{N}}

\newcommand{\ZZ}{\mathbf{Z}}

\newcommand{\balpha}{\boldsymbol{\alpha}}
\newcommand{\bbeta}{\boldsymbol{\beta}}
\newcommand{\bgamma}{\boldsymbol{\gamma}}
\newcommand{\bdelta}{\boldsymbol{\delta}}

\newcommand{\bx}{\boldsymbol{x}}
\newcommand{\bD}{\boldsymbol{D}}

\newcommand{\restr}{\mathord |}

\DeclareMathOperator{\im}{im}

\begin{document}

\title[Self-similarity of rational power series]{On the self-similarity of   rational power series with matrix coefficients}

\author{Pierre-Emmanuel \textsc{Caprace}}
\author{Justin \textsc{Vast}}
\date{May 21, 2026}

\thanks{JV is a F.R.S.-FNRS Research Fellow; PEC and JV are supported in part by the FWO and the F.R.S.-FNRS under the EOS programme (project ID 40007542).}

\begin{abstract}
Let $p$ be a prime, let $d \geq 1$ be an integer  and $A$ be the  algebra of square matrices of size~$d$ over the field of order $p$. 
Let $P, Q \in A[x_1, \dots x_n]$ be polynomials in $n$ indeterminates with coefficients in $A$, such that  $Q$ is invertible in $ A[\![x_1, \dots, x_n]\!]$.  Let also $\mathcal M \colon \mathbf Z^n \to A$ be the map associating to the $n$-tuple of integers $(\alpha_1, \dots, \alpha_n)$ the coefficient of the monomial $x_1^{\alpha_1} \dots x_n^{\alpha_n}$ in the development of the rational fraction $PQ^{-1}$ as a power series (the support of $\mathcal M$ is contained in $\mathbf N^n$). Our main result ensures that the map $\mathcal M$, viewed as a tiling of $\mathbf R^n$ by unit cubes with color set $A$, is self-similar. The self-similarity is expressed in terms of invariance under substitutions. By specializing to $d=1$, $n=2$, $P=1$ and $Q =1-x_1-x_2$, we recover the well-known self-similarity feature of the binomial coefficients modulo $p$.
\end{abstract}

\maketitle
 
\section{Introduction}
 
It has long been observed that the binomial coefficients have remarkable divisibility properties. A striking illustration is provided by the self-similar features of the array of numbers  defined by assigning to a pair of integers $(m, n) \in \mathbf N^2$, the residue class of the binomial coefficient $\binom{m+n} m$ modulo a fixed prime $p$, see \cite{Sved1} (the symbol $\mathbf N$ denotes the set of non-negative integers).   An arithmetic explanation of this phenomenon is afforded by a classical theorem of E.~Lucas from 1878 (see \cite[p. 271]{Lucas} and \cite{Fine}). Lucas' theorem has been extended and generalized in many ways, see \cite{Mestrovic} for a 2014 survey. Among those generalizations, we mention the following result, due to M.~Razpet, which is particularly relevant to our purposes:
 
 \begin{thm}[see  \cite{Razpet}]\label{thm:Razpet}
 Let $a, b, c \in \mathbf N$ and $w \colon \mathbf Z^2 \to \mathbf N$ be the map recursively defined by the following conditions:
 $w(0, 0) = 1$,  $ w(m, n) = 0$ if  $(m, n) \not \in \mathbf N^2$, and 
 \begin{align}\label{eq:rec}
   w(m, n) =  
  a w(m-1, n) + bw(m, n-1) + c w(m-1, n-1)
\end{align}
  if  $ (m, n)  \in \mathbf N^2 \setminus \{(0, 0)\}$.
  
Then, for any prime $p$, we have
$$w(p\alpha + \beta, p\gamma + \delta ) \equiv w(\alpha, \gamma) w ( \beta, \delta) \mod p$$
for all integers $\alpha, \beta, \gamma, \delta$ with $0\leq \beta, \delta < p$. 
  \end{thm}

M.~Razpet's theorem has been rediscovered a few years later in the field of molecular programming by Kautz--Lathrop~\cite{KaLa}. 
 
Taking $(a, b, c) = (1, 1, 0)$ in Theorem~\ref{thm:Razpet}, we have $w(m, n) = \binom{m+n} m$ for all $ (m, n)  \in \mathbf N^2$, and Lucas' theorem follows directly.  For any given triple $(a, b, c)$, the corresponding number $w(m, n)$ counts weighted lattice paths; it is known as a \textbf{Delannoy number} in the case $(a, b, c)=(1, 1, 1)$. 

In this note, we shall view the map assigning to $(m, n)$ the residue class of $w(m, n)$ modulo $p$ as a tiling of the Euclidean plane by unit squares, colored by the elements of the prime field $\FF_p$. More generally, a \textbf{$n$-dimensional tiling} with \textbf{color set} $A$ is defined as a map $\mathcal T \colon \ZZ^n \to  A$. Such a map may indeed be viewed as the standard tiling of $\mathbf R^n$ by unit cubes, each of which is colored by an element of $A$ via the map $\mathcal T$. 

The Lucas' property implies that the tiling $(m, n) \mapsto w(m, n) \mod p$ has fractal features, similar to those enjoyed by the binomial coefficients. 
A generalization of Razpet's theorem to $n$-dimensional tilings with any finite field as a color set, has been established in \cite{Pan}. The starting point of the present  work was the empirical observation that arrays of matrices with coefficients in a finite field, defined by a recurrence relation as in (\ref{eq:rec}), also enjoy remarkable self-similarity properties. We refer to Figure~\ref{fig1} below for an illustration (see also Figure~\ref{fig2bis} in the appendix).  After having observed the tilings depicted in Figure~\ref{fig1}, we were naturally led to the following questions: \textit{In what precise sense are those tilings self-similar? What is the algebraic explanation of this phenomenon?} The goal of this note is to provide a self-contained answer to those questions, that can already be epitomized as follows: the self-similarity can be described through the invariance under a so-called \textit{substitution}, and the algebraic source of that intrinsic symmetry is the Frobenius map.

\begin{figure}[h]
\includegraphics[width=6.8cm]{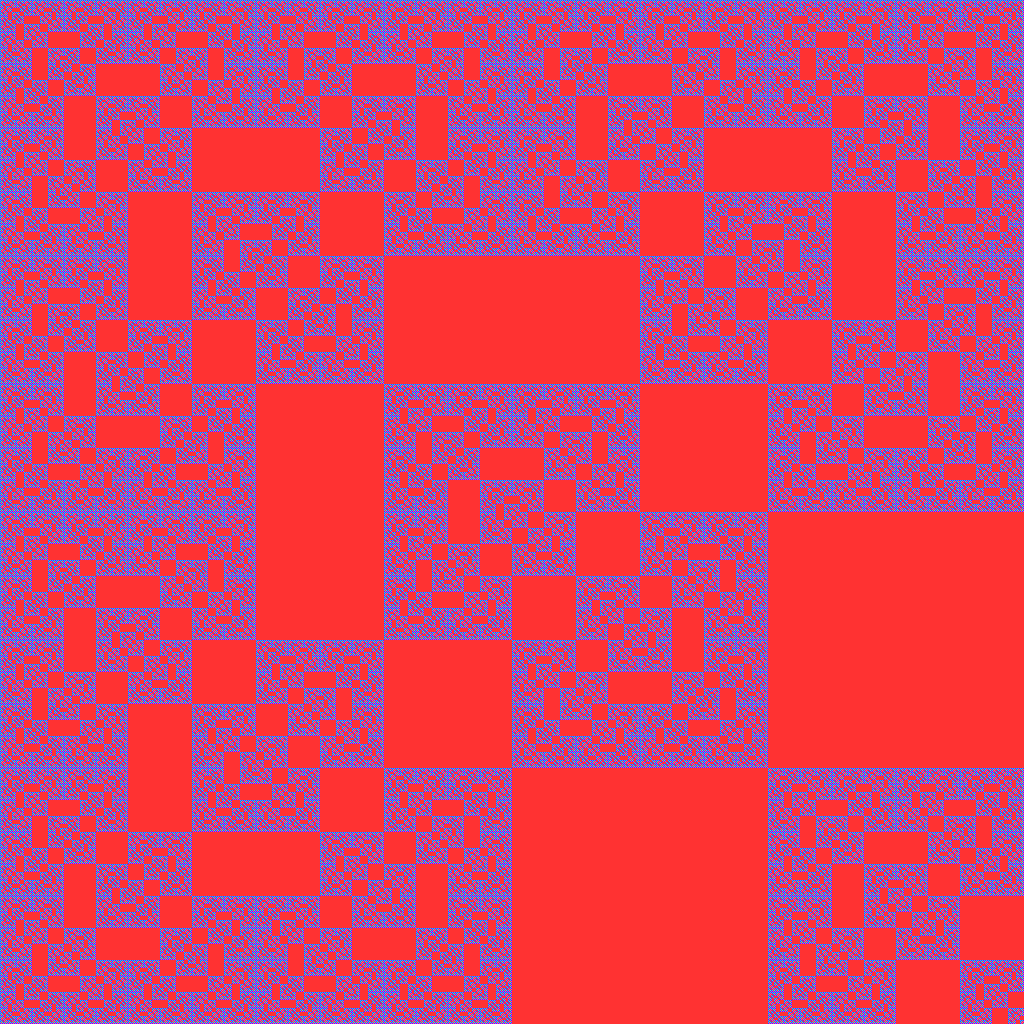} \hspace{.3cm}\includegraphics[width=6.8cm]{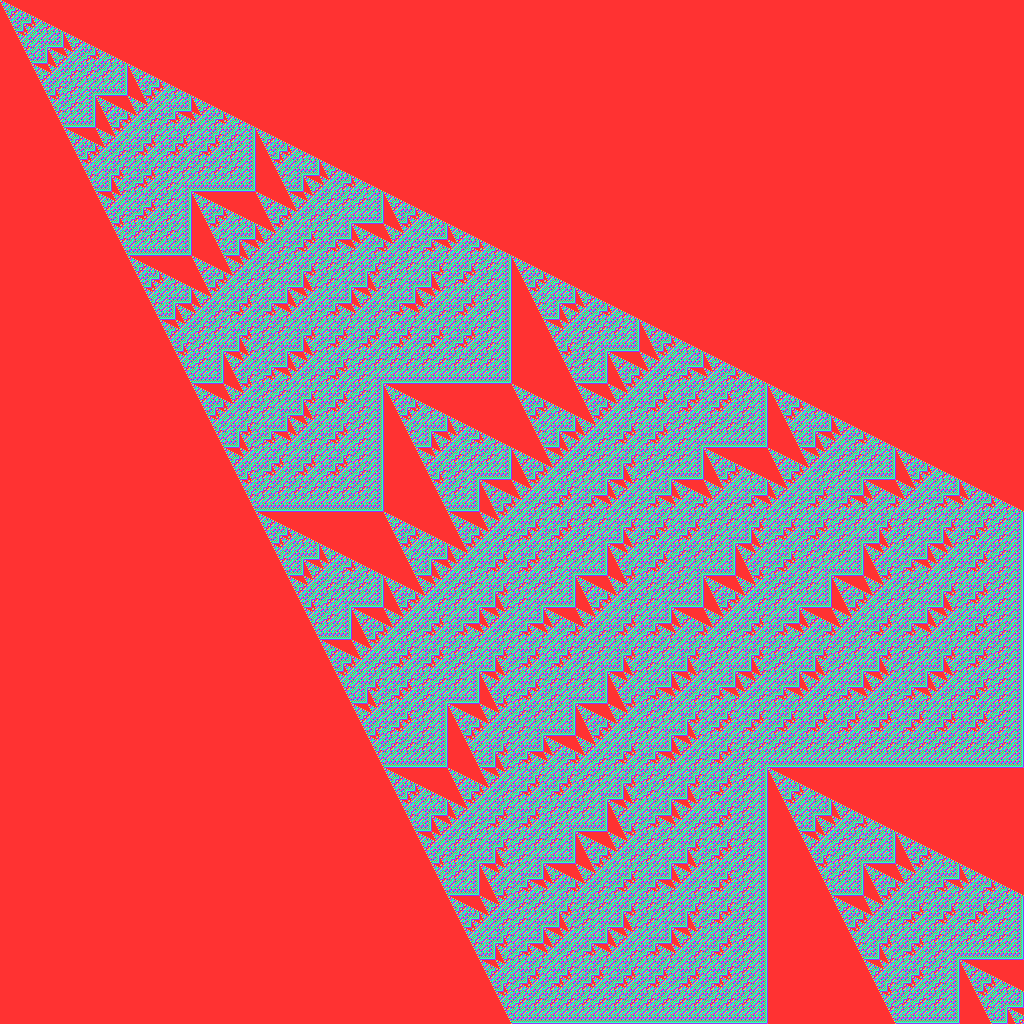}

\caption{Two illustrations of Corollary~\ref{cor:2-d} with $d=2$ and $p=2$. The pictures represent a portion of size $1024$ of the first octant; the origin is placed on the top left corner and the $y$-axis is oriented downward. Each color represents an element of $A$; the zero matrix corresponds to the red color.\\
On the left: $a = b = \left(\begin{array}{cc} 1 & 1 \\ 0 & 1\end{array}\right)$ and $c = \left(\begin{array}{cc} 1 & 1 \\ 1 & 0\end{array}\right)$. \\
On the right: $a = b = \left(\begin{array}{cc} 0 & 1 \\ 0 & 0\end{array}\right)$ and $c = \left(\begin{array}{cc} 1 & 1 \\ 1 & 0\end{array}\right)$.}
\label{fig1}
\centering
\end{figure}

Let us first review     the notion  of \textit{substi\-tution-invariance}. Sequences invariant under substitutions have appeared in symbolic dynamics (see e.g. \cite{Morse} and \cite{Gott}) and in information and communication theory (see e.g. \cite{Cobh}). From the point of view of tilings, sequences correspond to the one-dimensional case.  The idea of considering two- and higher-dimensional tilings that are invariant under substitutions  appears in the work of \v{C}ern{\'y}--Gruska \cite{CeGr} in theoretic computer science, and in the work of O.~Salon \cite{Sal1, Sal2}, that has a number theoretic flavor, and pertains to the study of \textit{automatic sequences} (see \cite[Chapter~14]{AlSh}). Salon's work is mentioned  by   Shallit--Stolfi \cite{ShSt} in a paper that  emphasizes the  fractals generated by two-dimensional substitutions.  Independently, and roughly at the same time,  two-dimensional substitutions appear in the work of S.~Mozes \cite{Moz}, that has a dynamical system flavor and is motivated by  the study of \textit{aperiodic tilings}. In this note, we define the notion of a substitution map  as follows. 
 
 Given sets $A$ and $E$, we let $A^E$ be the set of all maps $E \to A$. 
  Given an integer $\ell \geq 1$, we set $\mathcal I = \{0, 1, \dots, \ell-1\}$ and we define a \textbf{$n$-dimensional substitution of   length $\ell$} as a map
 $$\mathcal S \colon A \to  A^{\mathcal I^n } : a \mapsto \mathcal S_a.$$
 Given any  tiling $\mathcal T \in A^{\mathbf Z^n}$,  we define another tiling ${}^\mathcal S \mathcal T \in A^{\mathbf Z^n}$ by setting:
 $${}^\mathcal S  \mathcal T(\ell  \balpha +  \bbeta) = \mathcal S_{ \mathcal T(\balpha)}(   \bbeta)$$
 for all $\balpha \in \mathbf Z^n$ and $ \bbeta \in \mathcal I^n$. Hence,  the substitution $\mathcal S$ operates on the set $A^{\mathbf Z^n}$ of all tilings with color set $A$:   the new tiling is obtained by substituting   each tile of $\mathcal T$ with a cubic block of size $\ell$ prescribed by the map $\mathcal S$.  We say that $\mathcal T$ is \textbf{invariant under substitution} if there is a substitution $\mathcal S$ of   length $\ell \geq 2$ such that ${}^\mathcal S \mathcal T = \mathcal T$. 
 
 As an illustration of that notion, we record the following result, which  is a direct reformulation of Theorem~\ref{thm:Razpet}   in terms of substitutions, where $\mathbf F_p = \mathbf Z/p\mathbf Z$ denotes the field of order $p$.
 
 \begin{cor}\label{cor:Raz}
Retain the notation of Theorem~\ref{thm:Razpet}. Let $p$ be a prime and $\mathcal T \colon \mathbf Z^2 \to \mathbf F_p$ be the map assigning to $(m, n)$ the residue class of $w(m, n)$ modulo $p$. Then $\mathcal T$ is invariant under the $2$-dimensional substitution of length $p$ defined by 
$$\mathcal S \colon \mathbf F_p   \to  (\mathbf F_p)^ {\mathcal I^2}  : a \mapsto \big( \mathcal S_a \colon \bbeta \mapsto a \mathcal T(\bbeta) \big),$$
where $\mathcal I = \{0, 1, \dots, p-1\}$. 
 \end{cor}
 
 For $p=2$ and $(a, b, c) = (1, 1, 0)$ in Corollary~\ref{cor:Raz}, we have $\mathcal S_0 =  \left(\begin{array}{cc} 0 & 0 \\ 0 & 0\end{array}\right)$ and $\mathcal S_1 =  \left(\begin{array}{cc} 1 & 1 \\ 1 & 0\end{array}\right)$, where we have represented the elements of $(\mathbf F_2)^ {\mathcal I^2}$ as square matrices. 

The tilings depicted in Figure~\ref{fig1} have a matrix algebra as a color set; they are defined by a recursive condition similar to (\ref{eq:rec}). Importantly, such a tiling $\mathcal M$ need not be substitution-invariant in the strict sense, see Remark~\ref{rem:d=1} below. Our main result shows that $\mathcal M$ is the shadow of an explicitly defined tiling that is invariant under a $2$-dimensional substitution of   length $p$. This will moreover imply that the tiling obtained by viewing $\mathcal M$ as a tiling by cubic blocks of size $p^r$, is also invariant under a substitution.

In order to describe the result  accurately, we define, for each $\ell \geq 1$ and  any tiling $\mathcal T \in A^{\mathbf Z^n}$ with color set $A$  a tiling $\mathcal T^\ell$ of $\mathbf R^n$ by cubes of size $\ell$, with color set $A^{\mathcal I^n} $,  where $\mathcal I = \{0, 1, \dots, \ell-1\}$. Formally,    the map $\mathcal T^\ell \colon \mathbf Z^n \to A^{\mathcal I^n} $ is defined by setting 
$$\mathcal T^\ell \colon  \balpha \mapsto \big(  \mathcal I^n \to A \colon \bbeta \mapsto \mathcal T(\ell \balpha + \bbeta) \big).$$
Given any $\balpha \in \ZZ^n$ and any subset $\mathcal K \subseteq \ZZ^n$, we also define 
$$\mathcal T \restr_{\balpha + \mathcal K} \colon \mathcal K \to A : \bbeta \mapsto \mathcal T(\balpha + \bbeta).$$
Thus, for all $\balpha \in \ZZ^n$ and $\mathcal K \subseteq \ZZ^n$, we have a map
$$A^{\ZZ^n} \to A^{\mathcal K} : \mathcal T \mapsto \mathcal T \restr_{\balpha + \mathcal K}.$$
 Hence, for each $\ell \geq 1$, the tiling $\mathcal T^\ell$ associated with $\mathcal T$ as above may also be viewed as the tiling defined by
 $$\mathcal T^\ell (\balpha) = \mathcal T\restr_{\ell \balpha + \mathcal I^n}.$$ 

Given a polynomial $P$ in the indeterminates $x_1, \dots, x_n$, we write $\deg_{x_i}(P)$ for the degree of $P$ viewed as a polynomial in the sole indeterminate $x_i$, and we set $\deg(P) = \max \{\deg_{x_i}(P) \mid i=1, \dots, n\}$. We can now state the main result of this note.

\begin{thm}\label{thm:general}
Let $p$ be a prime,   $d \geq 1$ be an integer  and $A = \mathrm{Mat}_{d \times d}(\mathbf F_p)$  be the corresponding matrix algebra over $\FF_p$. Let also $P, Q \in A[x_1, \dots, x_n]$ be   polynomials such that $Q$ is invertible in $A[\![ x_1, \dots, x_n] \!]$. Let  $\mathcal M\colon \ZZ^n \to A$ be the tiling defined by the condition 
$$P {Q}^{-1} = \sum_{\balpha \in \ZZ^n} \mathcal M(\balpha) \bx^{\balpha},$$
where we have denoted  the monomial $ x_1^{\alpha_1} \dots x_n^{\alpha_n}$ by the symbol  $\bx^{\balpha}$, for any $\balpha  = (\alpha_1, \dots, \alpha_n) \in \ZZ^n$. Let $D = \max \{1, 1+\deg(P), d \deg(Q)\}$,  let $\mathcal J(0) = \{-D+1, \dots, -1, 0\}$ and  $W = (\FF_p)^{\mathcal J(0)^n}$. We  identify the algebra $A[\![ x_1, \dots, x_n] \!]$ of power series with matrix coefficients, with the algebra $\mathrm{Mat}_{d \times d}(\FF_p[\![ x_1, \dots, x_n] \!])$ of matrices with power series coefficients, so that $\det(Q) \in \FF_p[x_1, \dots, x_n]$ is an invertible element of $\FF_p[\![ x_1, \dots, x_n] \!]$. We let  $\mathcal T \colon \ZZ^n \to \FF_p$ be the tiling   defined by  
$$\frac 1 {\det(Q)} = \sum_{\balpha \in \ZZ^n} \mathcal T(\balpha) \bx^{\balpha},$$ 
 and  $\overline{\mathcal T} \colon \ZZ^n \to W$ be  the tiling defined by 
$$\overline{\mathcal T}(\balpha) = \mathcal T|_{\balpha + \mathcal J(0)^n}$$ for all $\balpha \in \ZZ^n$. 

The following assertions hold.
\begin{enumerate}[label=(\roman*)]
\item  The tiling $\overline{\mathcal T}$ is invariant under a linear substitution of   length $p$. 
\item There is a  $\FF_p$-linear map $\tau \colon W \to A$ such that $\mathcal M = \tau \circ \overline{\mathcal T}$. 
\item There exist non-negative integers $r, t$  with $t\geq 1$ such that the tiling 
$$\mathcal M^{p^r} \colon \mathbf Z^n \to A^{\mathcal I(r)^n}$$ 
is invariant under a $n$-dimensional  substitution  of length $p^t$, where $\mathcal I(r) = \{0, 1, \dots, p^r-1\}$. Moreover the substitution map is linear of rank at most~$D^n$. 
\end{enumerate}
A similar result holds for $Q^{-1}P$.
 \end{thm}
 
 The condition that the polynomial $Q$ be invertible as a power series is equivalent to the condition that its independent term be an invertible matrix, see Lemma~\ref{lem:invertible-power-series} below. 
  
 \begin{remark}\label{rem:Salon}
Following the terminology of O.~Salon \cite{Sal1},   we say that   $\mathcal M \colon \ZZ^n \to A$  with color set $A$ is \textbf{generated by a substitution} of length $p$ if there is a set $W$, a tiling $\mathcal M'$ with color set $W$, invariant under a substitution of length $p$,  and a map $\tau \colon W \to A$ such that $\mathcal M = \tau \circ \mathcal M'$. Note that the tiling $\mathcal M$ itself need not be substitution-invariant. The fact that the tiling $\mathcal M$ from Theorem~\ref{thm:general} is indeed generated by a substitution of length $p$ may be viewed as a consequence of  the main result of \cite{Sal1}. Theorem~\ref{thm:general} specifies explicitly how to construct the substitution-invariant tiling $\mathcal M'$ (namely $\mathcal M' = \overline{\mathcal T}$). Moreover, it underlines the linearity of all the maps involved, and it highlights  the substitution-invariance of the tiling $\mathcal M^{p^r}$. The proof of Theorem~\ref{thm:general} presented below is self-contained. 
 \end{remark}

\begin{remark}\label{rem:support}
Observe that, by definition, the support of the map $\mathcal M$ from Theorem~\ref{thm:general} is contained in $\mathbf N^n$, so that $\mathcal M$ can rather be viewed as a tiling of the first octant. Thus $\mathcal M$ could be viewed as a multi-index sequence. This is the point of view adopted by O.~Salon \cite{Sal1, Sal2}; see also \cite[Chapter~14]{AlSh}, where the term \textbf{multidimensional automatic sequence} is used. In this note, we have rather followed the point of view of tilings defined over the entire Euclidean $n$-space, which is customary in substitution dynamics. In Theorem~\ref{thm:general}, it is indeed important to view $\mathcal M$ and $\mathcal T$ as maps defined on the whole of $\ZZ^n$, otherwise the definition of $\overline{\mathcal T}$ would not make sense for $D>1$. 
\end{remark}

 By setting $n=2$, $P=1$ and $Q = 1 - (ax + by + cxy)$ in Theorem~\ref{thm:general}, we derive the following.

\begin{cor}\label{cor:2-d}
Let $p$ be a prime,   $d \geq 1$ be an integer  and $A = \mathrm{Mat}_{d \times d}(\FF_p)$ be the corresponding matrix algebra over $\FF_p$. Given matrices $a, b, c \in A$, we let $\mathcal M \colon \mathbf Z^2 \to A$ be the tiling recursively defined by the following conditions:
$\mathcal M(0, 0) = 1$,  $ \mathcal M(m, n) = 0$ if  $(m, n) \not \in \mathbf N^2$, and 
 \begin{align}\label{eq:rec:2}
   \mathcal M(m, n) =  
  a \mathcal M(m-1, n) + b\mathcal M(m, n-1) + c \mathcal M(m-1, n-1)
\end{align}
  if  $ (m, n)  \in \mathbf N^2 \setminus \{(0, 0)\}$.
  
Then there is a $2$-dimensional tiling $\overline{\mathcal T}$ with color set $A$, invariant under a linear substitution of length $p$, and a linear map $\tau \colon A \to A$ such that $\mathcal M = \tau \circ \overline{\mathcal T}$.  
\end{cor}

Indeed, the generating series $G(x, y) = \sum_{(m, n) \in \mathbf N^2} \mathcal M(m, n) x^m y^n$  is an element of the $\mathbf F_p$-algebra  $A[\![x, y]\! ]$ of power series in two inderminates with coefficients in the matrix algebra $A$. The recurrence relation (\ref{eq:rec:2}) ensures that  $G = (1-Q)^{-1}$, where $Q(x, y)  = ax + by +c xy$ is a polynomial in two indeterminates with matrix coefficients. Observe that $\deg(Q)=1$ and $D = d$ in this case, so that $W$ may naturally be identified with $A$. Thus Theorem~\ref{thm:general} applies to the rational power series $G$, and Corollary~\ref{cor:2-d} follows.

\begin{remark}\label{rem:d=1->Razpet}
In the special case where $d=1$ in Corollary~\ref{cor:2-d}, we have $\overline{\mathcal T} = \mathcal T= \mathcal M$ in the notation of Theorem~\ref{thm:general}. This shows that $\mathcal M$ is invariant under a substitution of length $p$ in that case. Thus we recover Corollary~\ref{cor:Raz} as a special case of Corollary~\ref{cor:2-d}.
\end{remark}

\begin{remark}\label{rem:d=1}
In the special case where $d=1$ and $\deg (Q) \leq 1$, the tiling $\mathcal M$ from Theorem~\ref{thm:general} was also considered by H.~Pan~\cite{Pan} for arbitrary values of $n$. The main result of \cite{Pan} is that  $\mathcal M$ satisfies a Lucas' type property in this case. This directly implies that this tiling $\mathcal M$ is invariant under a $n$-dimensional substitution. The latter fact also follows from Theorem~\ref{thm:general}, in the same way as Corollary~\ref{cor:2-d} and Remark~\ref{rem:d=1->Razpet}. However, even for $d=1$, the tiling $\mathcal M$ need not be invariant under a substitution   if the degree of $Q$ is allowed to be larger than $1$. Indeed, if $\mathcal M$ were invariant under a substitution of  length $\ell$, then the tiling $\mathcal M^\ell$ would have at most $p$ colors. The tilings depicted in Figure~\ref{fig2} show that this is not the case in general. However, Theorem~\ref{thm:general} ensures that the tiling $\mathcal M^{p^r} $ is substitution-invariant for some $r \geq 0$. 
\end{remark}

Observe the striking similarity between the pictures represented in Figure~\ref{fig1} and~\ref{fig2}: the patterns of zeros look identical in the corresponding tilings, while the color set is different: in Figure~\ref{fig2} the tilings involve only two colors. This illustrates well the conclusions of Theorem~\ref{thm:general}, that relate the tiling $\mathcal M$ to a tiling $\mathcal T$: the tilings in Figure~\ref{fig1} correspond to $\mathcal M$ in Theorem~\ref{thm:general}, while those in Figure~\ref{fig2} correspond to $\mathcal T$. For the proof of Theorem~\ref{thm:general}, we shall  first treat the case where  $d=1$, see Proposition~\ref{prop:P/1-Q} below. The identification of the algebra of power series $A[\![ x_1, \dots, x_n] \!]$ with $\mathrm{Mat}_{d \times d}(\mathbf F_p[\![ x_1, \dots, x_n] \!])$ will allow us to treat one component of the matrix $\mathcal M(\balpha)$ at a time, and thereby reduce the general case of Theorem~\ref{thm:general} to the case where $d=1$.

\begin{figure}[h]
\includegraphics[width=6.8cm]{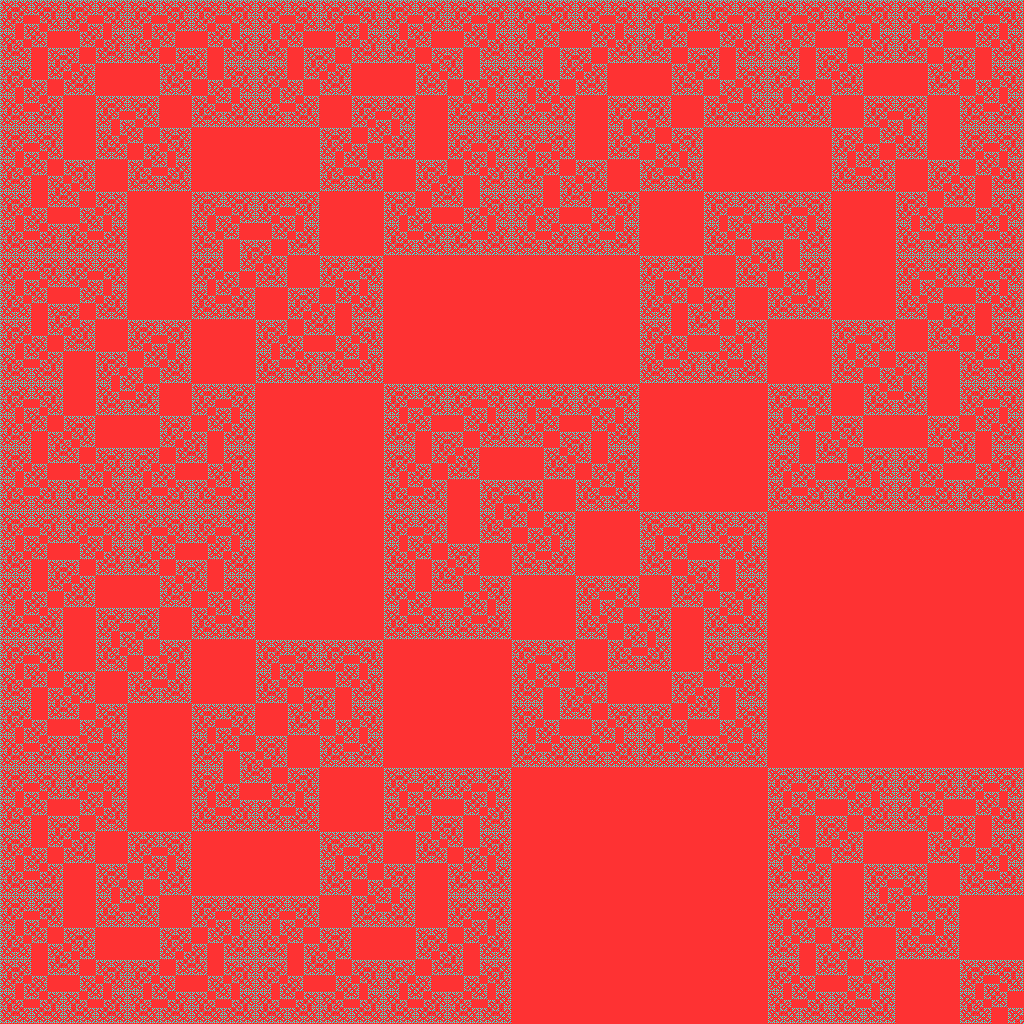} \hspace{.3cm}\includegraphics[width=6.8cm]{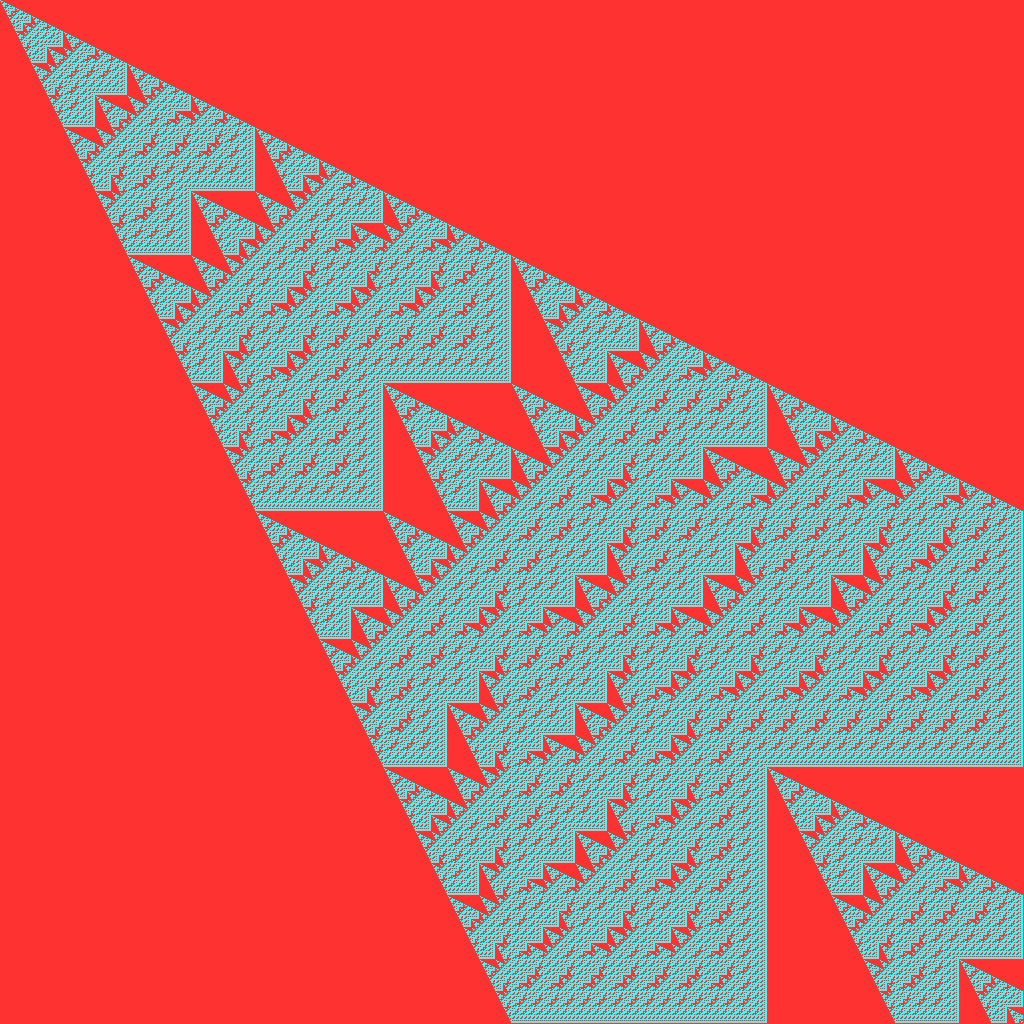}

\caption{Two illustrations of Theorem~\ref{thm:general} with $d=1$ and $p=2$. The pictures represent a portion of size $1024$ of the first octant; the origin is placed on the top left corner and the $y$-axis is oriented downward. Each color represents an element of $A = \FF_2$; the zero element corresponds to the red color.\\
On the left: $P=1$ and $Q = 1+x_1^2  + x_2^2 + x_1x_2 + x_1^2x_2^2$, \\
On the right: $P=1$ and $Q = 1 +x_1^2x_2 + x_1x_2^2 + x_1x_2 + x_1^2x_2^2$.}
\label{fig2}
\centering
\end{figure}

\begin{remark}\label{rem:Frobenius}
In the all the related works by E.~Lucas~\cite{Lucas} (see also  \cite{Fine}), O.~Salon~\cite{Sal1}, M.~Razpet \cite{Razpet} and H.~Pan~\cite{Pan} mentioned above,  as well as in Theorem~\ref{thm:general}, the algebraic source of self-similiarity, incarnated in this note by a certain substitution of length $p$, is the Frobenius map $\lambda \mapsto \lambda^p$, that defines an endomorphism of the underlying commutative algebras. That Frobenius map appears explicitly in the proof of Proposition~\ref{prop:Frob} below. 
\end{remark}

\begin{remark}
Theorem~\ref{thm:general} directly implies a similar statement for matrix algebras over arbitrary finite fields. Indeed, given a  field extension $F$ of degree $\ell$ of $\FF_p$, the matrix algebra $\mathrm{Mat}_{d \times d}(F)$ embeds as a subalgebra of $\mathrm{Mat}_{d\ell \times d\ell }(\mathbf F_p)$.
\end{remark}

\begin{remark}
Substitution-invariant tilings, as well as tilings generated by substitutions, may be viewed as discrete fractals. After a suitable rescaling process, they also give rise to non-discrete compact fractal sets. Those have been extensively studied, see for example \cite{BaHa} and \cite{Grein}. 
\end{remark}

\begin{remark}
Dynamical systems associated with higher-dimensional substitutions have also been studied extensively. We refer to the survey \cite{Frank} and the references cited therein. Until recently, most of the known results focused on the case where the substitution map satisfies a property named  \textit{primitivity}, which basically means that every color appears in every cubic block defined by the substitution (see Section~2.2 in \cite{Frank} for the formal definition). Since the substitutions appearing in Theorem~\ref{thm:general} and Corollary~\ref{cor:2-d} are linear maps, they cannot be primitive, since they map the zero matrix to a cubic block only colored by zero.  The dynamical study of substitution tilings in the imprimitive case has been undertaken in \cite{CoSo}. 
\end{remark}
	
More pictures illustrating the main result of this note are included in an appendix.

\section{A criterion for the substitution-invariant by blocks}

 Throughout this paper, we fix a prime number $p$.  For each $r \geq 0$, we set 
$$\mathcal I(r) = \{0, 1, \dots, p^r-1\}.$$
The following auxiliary fact will serve as a tool that yields the conclusion of Theorem~\ref{thm:general}(iii).

\begin{prop}\label{prop:tool}
Let $V, W$ be  finite-dimensional vector spaces over $\FF_p$, and let $\mathcal M \colon \ZZ^n \to V$ and $ \mathcal L \colon \ZZ^n \to W$ be tilings with color sets $V$ and $W$ respectively. Suppose that for each integer $s \geq 0$, there is a linear map $\Phi_s \colon W \to V^{\mathcal I(s)^n}$ such that for all $\balpha \in \ZZ^n$, we have 
$$\mathcal M\restr_{p^s \balpha + \mathcal I(s)^n} = \Phi_s(\mathcal L(\balpha)).$$
Then there exist integers $r \geq 0$ and $t \geq 1$ such that $\mathcal M^{p^r}$ is invariant under a $n$-dimensional linear substitution of   length $p^t$. Moreover, the substitution map has rank at most~$\dim(W)$.

If in addition $\Phi_s$ is injective for some $s\geq 0$, then we can take $r=s$ and $t=1$.
\end{prop}

\begin{proof}
By hypothesis,   the  vector space $W$ is finite. Hence there exist $r \geq 0$ and $r' >r$ such that $\ker(\Phi_r) \leq \ker(\Phi_{r'})$. If in addition we have $\ker(\Phi_s) = \{0\}$ for some $s \geq 0$, then we may further assume that $r=s$ and $r' = r+1$. 

For each $s \geq 0$, we set  $W_s = \im(\Phi_s)$ and $\Phi'_s \colon W/\ker(\Phi_s) \to W_s$ be the canonical linear isomorphism such that $\Phi_s(\mathcal L(\balpha)) = \Phi'_s(\mathcal L(\balpha) + \ker(\Phi_s)) $ for all $\balpha$. Since $\ker(\Phi_r) \leq \ker(\Phi_{r'})$, the canonical project $W \to W/\ker(\Phi_{r'})$ factors through $W/\ker(\Phi_{r})$. Hence there exists a linear map $\rho \colon W_r \to W_{r'}$ such that $\Phi_{r'} = \rho \circ \Phi_r$. Therefore, for all $\balpha \in \ZZ^n$, we have
\begin{align*}
\mathcal M\restr_{p^{r'}\balpha + \mathcal I(r')^n} 
& = \Phi_{r'}(\mathcal L(\balpha))\\
&=\Phi'_{r'}(\mathcal L(\balpha) + \ker(\Phi_{r'}))\\ 
&= \rho \circ \Phi_r(\mathcal L(\balpha) )\\ 
&= \rho(\mathcal M\restr_{p^{r}\balpha + \mathcal I(r)^n} ).
\end{align*}
We choose a linear  extension of $\rho$ as a linear map $V^{\mathcal I(r)^n} \to V^{\mathcal I(r')^n}$.  
Recall from   that $\mathcal M^{p^r}$ is the tiling with color set $V^{\mathcal I(r)^n} $ defined as
$\mathcal M^{p^r}(\balpha) = \mathcal M\restr_{p^{r}\balpha + \mathcal I(r)^n}$.
It follows that for all  $\balpha \in \ZZ^n$, we have 
\begin{align*}
\mathcal M^{p^{r'}}(\balpha) 
&= \mathcal M\restr_{p^{r'}\balpha + \mathcal I(r')^n}  \\
&= \rho(\mathcal M\restr_{p^{r}\balpha + \mathcal I(r)^n} ) \\
&= \rho(\mathcal M^{p^{r}}(\balpha)).
\end{align*}
We set $t = r' -r$. To each $f \in V^{\mathcal I(r')^n}$, we associate a map $\tilde f \in (V^{\mathcal I(r)^n})^{\mathcal I(t)^n}$ by setting
$$\tilde f (\bbeta)(\bdelta) = f(p^r \bbeta +\bdelta)$$
for any $\bbeta \in \mathcal I(t)^n$ and $\bdelta \in \mathcal I(r)^n$. 
In this way, we obtain a linear isomorphism $\sigma \colon V^{\mathcal I(r')^n} \to (V^{\mathcal I(r)^n})^{\mathcal I(t)^n} : f \mapsto \tilde f$. Finally, we define $\mathcal S \colon V^{\mathcal I(r)^n} \to (V^{\mathcal I(r)^n})^{\mathcal I(t)^n} $ by setting $\mathcal S = \sigma \circ \rho$. Observe that $\mathcal S$ is a linear map. 
Now, for all $\balpha \in \ZZ^n$,  $\bbeta \in  \mathcal I(t)^n$ and $\bdelta \in \mathcal I(r)^n$, we have 
\begin{align*}
\mathcal M^{p^{r}}(p^t \balpha + \bbeta)(\bdelta)
&= \mathcal M(p^r p^t \balpha + p^r \bbeta +\bdelta)\\
&=  \mathcal M(p^{r'}  \balpha + p^r \bbeta +\bdelta)\\
&= \mathcal M^{p^{r'}}(\balpha)(p^r \bbeta +\bdelta)\\
&= \rho(\mathcal M^{p^{r}}(\balpha))(p^r \bbeta +\bdelta)\\
&=\mathcal S_{\mathcal M^{p^{r}}(\balpha)}(\bbeta)(\bdelta).
\end{align*}
Hence $\mathcal M^{p^{r}}(p^t \balpha + \bbeta) = \mathcal S_{\mathcal M^{p^{r}}(\balpha)}(\bbeta)$. 
This confirms that the tiling $\mathcal M^{p^r}$ is invariant under the linear substitution $\mathcal S$, which is  of   length $p^t$. By construction the rank of $\mathcal S$ is bounded above by the dimension of $W$. 
\end{proof}
	
\begin{cor}\label{cor:tool}
Let $V, W$ be  finite-dimensional vector spaces over $\FF_p$, let $ \mathcal L \colon \ZZ^n \to W$ be a tiling with color set $W$. Let also $\tau \colon W  \to V$ be a linear map and $\mathcal M \colon \ZZ^n \to V$ be the tiling defined  	as $\mathcal M = \tau \circ \mathcal L$. 

If  $\mathcal L$ is invariant under a linear substitution $\mathcal S$ of length $p$, then there exist integers $r \geq 0$ and $t \geq 1$ such that $\mathcal M^{p^r}$ is invariant under a $n$-dimensional linear substitution of   length $p^t$. Moreover, the substitution map for $\mathcal M^{p^r}$ has rank at most~$\dim(W)$.
\end{cor}
\begin{proof}
By hypothesis, we have $\mathcal L|_{p \balpha + \mathcal I(1)^n} = \mathcal S_{\mathcal L(\balpha)}$, where $\mathcal S \colon W \to W^{\mathcal I(1)^n}$ is a linear substitution of length $p$. Given $s \geq 1$, we write $\mathcal S^s$ for the map $W \to W^{\mathcal I(s)^n}$ obtained by iterating $s$ times the substitution  $\mathcal S$. We obtain  $\mathcal L|_{p^s \balpha + \mathcal I(s)^n} = \mathcal S^s_{\mathcal L(\balpha)}$ for all $\balpha \in \ZZ^n$. For each non-empty set $\mathcal K$, the linear map $\tau \colon W \to V$ induces a linear map $\tau^{\mathcal K} \colon W^{\mathcal K} \to V^{\mathcal K}$ defined by applying $\tau$ componentwise.  Setting $\Phi_s = \tau^{\mathcal I(s)^n} \circ \mathcal S^s$, we see that the hypotheses of Proposition~\ref{prop:tool} are satisfied. The conclusion follows. 
\end{proof}

\section{Rational fractions with scalar coefficients}

The following basic fact will often be used without further notice. 

\begin{lem}\label{lem:invertible-power-series}
Let $A$ be an associative ring with a unit. A power series $Q \in A[\![ x_1, \dots, x_n]\!]$ is invertible in $A[\![ x_1, \dots, x_n]\!]$  if and only if the independent term of $Q$ is invertible in $A$. 
\end{lem}
\begin{proof}
Let $J$ be the ideal of $A[\![ x_1, \dots, x_n]\!]$ generated by all the indeterminates. The quotient map $A[\![ x_1, \dots, x_n]\!] \to A[\![ x_1, \dots, x_n]\!] /J \cong A$ sends a power series to its independent term. It follows that if $Q$ is invertible, then its independent term is invertible as well. 

Conversely, suppose that  the independent term of $Q$, denoted by $a$, is invertible. Set $R=1- a^{-1}Q$. Observe that the independent term of $R$ is zero.  Then for all $m \geq 1$ we have $R^m \in J^m$. It follows that  the series $S = \sum_{m \geq 0} R^m$ converges in $A[\![ x_1, \dots, x_n]\!]$. We have $(1-R)^{-1} = S$ so that  $1-R = a^{-1}Q$ is invertible. Thus $Q = a (a^{-1}Q)$ is invertible as well.  
\end{proof}

Let $R \in \FF_p[x_1, \dots, x_n]$ be a non-zero polynomial with a zero independent term. We define $\mathcal T\colon \ZZ^n \to \FF_p$ by the condition 
$$\frac 1 {1- R} = \sum_{\balpha \in \ZZ^n} \mathcal T(\balpha) \bx^{\balpha}.$$
This makes sense in view of Lemma~\ref{lem:invertible-power-series}.

We also define  $ h \colon \ZZ^n \to \FF_p$ by setting 
$$\sum_{i=0}^{p-1} R^i = \sum_{\bdelta \in \ZZ^n} h(\bdelta) \bx^{\bdelta}.$$
Observe that $h(\bdelta) = 0$ for all $\bdelta \not \in \NN^n$, and for almost all $\bdelta  \in \NN^n$. 
  
\begin{prop}\label{prop:Frob}
For all $\balpha \in \ZZ^n$, we have 
$$\mathcal T(\balpha) = \sum_{\substack{\bgamma, \bdelta \in \ZZ^n\\ p\bgamma + \bdelta = \balpha}} h(\bdelta) \mathcal T(\bgamma).$$
\end{prop}
\begin{proof}
The Frobenius map $\lambda \mapsto \lambda^p$ acts trivially on $\FF_p$. Therefore, we have 
{\allowdisplaybreaks
\begin{align*}
\sum_{\balpha \in \ZZ^n} \mathcal T(\balpha) \bx^{\balpha} 
&=\frac 1 {1- R} \\
&  = \sum_{i=0}^\infty R^i\\
& =(\sum_{i=0}^{p-1} R^i) (\sum_{j=0}^\infty R^{jp})\\
& =(\sum_{i=0}^{p-1} R^i) (\sum_{j=0}^\infty R^{j})^p\\
& =(\sum_{i=0}^{p-1} R^i)( \sum_{\bgamma \in \ZZ^n} \mathcal T(\bgamma) \bx^{\bgamma})^p\\
& =(\sum_{\bdelta \in \ZZ^n} h(\bdelta) \bx^{\bdelta})( \sum_{\bgamma \in \ZZ^n} \mathcal T(\bgamma)^p \bx^{p\bgamma})\\
& = \sum_{\bgamma, \bdelta \in \ZZ^n} h(\bdelta) \mathcal T(\bgamma) \bx^{p\bgamma + \bdelta}\\
& = \sum_{\alpha \in \ZZ^n} \big(\sum_{\substack{\bgamma, \bdelta \in \ZZ^n\\ p\bgamma + \bdelta = \balpha}} h(\bdelta) \mathcal T(\bgamma)\big) \bx^{\balpha}.
\end{align*}
}
The required conclusion follows. 	
\end{proof}

Let us fix an integer $D$ such that $D\geq  \deg(R)$. 
For each $r \geq 0$, we set 
$$ \mathcal J(r) = \{-D+1, \dots, -1, 0, 1, \dots, p^r-1\}.$$
Observe that $\mathcal I(r) \subset \mathcal J(r)$. 

\begin{prop}\label{prop:linear:d=1}
Let $A = \FF_p$. There exists a linear map $\Phi_1 \colon A^{\mathcal J(0)^n} \to A^{\mathcal J(1)^n}$ such that, for all $\balpha \in \ZZ^n$, we have
$$\mathcal T\restr_{p \balpha + \mathcal J(1)^n} = \Phi_1(\mathcal T\restr_{\balpha + \mathcal J(0)^n}).$$
\end{prop}
\begin{proof}
Given $\balpha \in \ZZ^n$ and $\bbeta \in \mathcal J(1)^n$, and invoking Proposition~\ref{prop:Frob}, we obtain successively
\begin{align*}
\mathcal T \restr_{p\balpha + \mathcal J(1)^n}(\bbeta) 
& = \mathcal T(p\balpha + \bbeta)\\
& =  \sum_{\substack{\bgamma, \bdelta \in \ZZ^n\\ p\bgamma + \bdelta = p\balpha + \bbeta}} h(\bdelta) \mathcal T(\bgamma)\\
&= \sum_{\substack{ \bdelta \in \ZZ^n\\ \bdelta \equiv \bbeta \mod p}}  h(\bdelta) \mathcal T(\balpha + \frac{ \bbeta-\bdelta} p),
\end{align*}
where we have written $\bdelta \equiv \bbeta \mod p$ to mean that $\delta_i \equiv \beta_i \mod p$ for all $i \in \{1, \dots, n\}$. Given  $\balpha, \bbeta \in \ZZ^n$, we also write $\balpha \geq \bbeta$ if $\alpha_i \geq \beta_i$ for all $i$. 
By definition, the coefficient $h(\bdelta)$ is zero, unless we have 
$$0 \leq \bdelta \leq (p-1) \bD, $$
where $\bD = (D, D, \dots, D)$.  
Consider such a $\bdelta$. Then, for each  $i \in \{1, \dots, n\}$, we have $1-D \leq \beta_i \leq p-1$ since $\bbeta \in \mathcal J(1)$. Therefore,  we have
$$\frac 1 p -D =  \frac{1-pD} p  = \frac{(1-D) - (p-1)D} p \leq \frac{\beta_i -\delta_i} p \leq \frac{p-1} p.$$
In the case where $\frac{\beta_i -\delta_i} p$ is an integer, we obtain 
$$1- D \leq \frac{\beta_i -\delta_i} p \leq 0.$$
Therefore, if   $0 \leq \bdelta \leq (p-1)\bD$ and $\bdelta \equiv \bbeta \mod p$, then we have $\frac{ \bbeta-\bdelta} p \in \mathcal J(0)^n$. Hence 
\begin{align*}
\mathcal T \restr_{p\balpha + \mathcal J(1)^n}(\bbeta)  
&= \sum_{\substack{0 \leq  \bdelta  \leq (p-1)\bD\\ \bdelta \equiv \bbeta \mod p}}  h(\bdelta) \mathcal T(\balpha + \frac{ \bbeta-\bdelta} p)\\
&=\sum_{\substack{0 \leq  \bdelta  \leq (p-1)\bD\\ \bdelta \equiv \bbeta \mod p}}  h(\bdelta) \mathcal T\restr_{\balpha + \mathcal J(0)^n}(\frac{ \bbeta-\bdelta} p).
\end{align*}
We define $\Phi_1 \colon A^{\mathcal J(0)^n} \to A^{\mathcal J(1)^n}$ by setting 
$$\Phi_1(f)(\bbeta) = \sum_{\substack{0 \leq  \bdelta  \leq (p-1)\bD\\ \bdelta \equiv \bbeta \mod p}}  h(\bdelta)f(\frac{ \bbeta-\bdelta} p)$$ 
for all $f \in A^{\mathcal J(0)^n}$ and $\bbeta \in \mathcal J(1)^n$.  Observe that $\Phi_1$ is linear. The calculations above show that $\mathcal T\restr_{p \balpha + \mathcal J(1)^n} = \Phi_1(\mathcal T\restr_{\balpha + \mathcal J(0)^n})$. 
\end{proof}

\begin{remark}
One may in fact show by induction  that for each $s \geq 0$, there exists a linear map $\Phi_s \colon A^{\mathcal J(0)^n} \to A^{\mathcal J(s)^n}$ such that, for all $\balpha \in \ZZ^n$, we have $\mathcal T\restr_{p ^s \balpha + \mathcal J(s)^n} = \Phi_s(\mathcal T\restr_{\balpha + \mathcal J(0)^n})$. 
Moreover  $\ker(\Phi_s) \leq \ker(\Phi_{s+1})$ for all $s$. We do not include a proof as we won't need that fact here. 
\end{remark}

\begin{cor}\label{cor:substitution-length-p}
Retain the notation of Proposition~\ref{prop:linear:d=1}. Set $W = A^{\mathcal J(0)^n}$ and let $\overline{\mathcal T} \colon \ZZ^n \to W$ be the tiling with color set $W$ defined by setting $\overline{\mathcal T}(\balpha) = \mathcal T|_{\balpha + \mathcal J(0)^n}$. Then $\overline{\mathcal T}$ is invariant under a linear substitution $\mathcal S \colon W \to W^{\mathcal I(1)^n}$ of   length $p$. 
\end{cor}
\begin{proof}
Let $\Phi_1 \colon W = A^{\mathcal J(0)^n} \to A^{\mathcal J(1)^n} $ be the linear map afforded by Proposition~\ref{prop:linear:d=1}. The latter result ensures that  $\mathcal T\restr_{p \balpha + \mathcal J(1)^n} = \Phi_1(\mathcal T\restr_{\balpha + \mathcal J(0)^n})$ for all $\balpha \in \ZZ^n$. For each $\bbeta \in \mathcal I(1)^n$, we have $\bbeta + \mathcal J(0)^n \subseteq \mathcal J(1)^n$. Therefore, we have a  linear map $\sigma  \colon A^{\mathcal J(1)^n} \to  (A^{\mathcal J(0)^n})^{\mathcal I(1)^n} =
W^{\mathcal I(1)^n}$ defined by 
$\sigma(f)(\bbeta) \colon \bgamma \mapsto f(\bbeta + \bgamma)$ for all $f \in A^{\mathcal J(1)^n}$, $\bbeta \in \mathcal I(1)^n$ and $\bgamma \in \mathcal J(0)^n$. 
Given $\balpha \in \ZZ^n$ and $\bbeta \in \mathcal I(1)^n$, we obtain
\begin{align*}
\overline{\mathcal T}(p\balpha + \bbeta) 
&=  \mathcal T|_{p\balpha + \bbeta + \mathcal J(0)^n}\\
& = \sigma(\mathcal T|_{p\balpha +   \mathcal J(1)^n})(\bbeta)\\
&= \sigma \circ \Phi_1(\mathcal T|_{\balpha +   \mathcal J(0)^n})(\bbeta)\\
& = \sigma \circ \Phi_1(\overline{\mathcal T}(\balpha))(\bbeta).
\end{align*}
Thus the tiling $\overline{\mathcal T}$ is invariant under the substitution $\mathcal S = \sigma \circ \Phi_1$, which is linear, and of   length $p$.
\end{proof}

We are now ready to prove Theorem~\ref{thm:general} in the case where $d=1$. We reproduce its statement for the reader's convenience.

\begin{prop}\label{prop:P/1-Q}
Let $p$ be a prime and set $A = \FF_p$. Let also $P, Q \in A[x_1, \dots, x_n]$ be  polynomials, and assume that $Q$ has a non-zero independent term (so that $Q$ is invertible as a power series by Lemma~\ref{lem:invertible-power-series}). We  let  $\mathcal M\colon \ZZ^n \to A$ be the tiling defined by the condition 
$$\frac P {Q} = \sum_{\balpha \in \ZZ^n} \mathcal M(\balpha) \bx^{\balpha}.$$
Let $D = \max \{1, 1+\deg(P), \deg(Q)\}$,  let $\mathcal J(0) = \{-D+1, \dots, -1, 0\}$ and  $W = A^{\mathcal J(0)^n}$. We also let   $\mathcal T \colon \ZZ^n \to A$ be the tiling defined by setting 
$$\frac 1 Q = \sum_{\balpha \in \ZZ^n} \mathcal T(\balpha) \bx^{\balpha},$$
   and $\overline{\mathcal T} \colon \ZZ^n \to W$ be  the tiling defined by $\overline{\mathcal T}(\balpha) = \mathcal T|_{\balpha + \mathcal J(0)^n}$ for all $\balpha \in \ZZ^n$. 

The following assertions hold.
\begin{enumerate}[label=(\roman*)]
\item  The tiling $\overline{\mathcal T}$ is invariant under a linear substitution of   length $p$. 
\item There is a linear form $\tau \colon W \to A$ such that $\mathcal M = \tau \circ \overline{\mathcal T}$. 
\item There exist non-negative integers $r, t$  with $t\geq 1$ such that the tiling 
$$\mathcal M^{p^r} \colon \mathbf Z^n \to A^{\mathcal I(r)^n}$$ 
is invariant under a $n$-dimensional  substitution $\mathcal S$ of length $p^t$, where $\mathcal I(r) = \{0, 1, \dots, p^r-1\}$. Moreover the substitution map is linear of rank at most~$D^n$. 
\end{enumerate}
\end{prop}
	
\begin{proof}
Let $a \in \FF_p^*$ be the independent term of $Q$. We have $\frac 1 Q = a^{-1} \frac{1}{1-(1-a^{-1}Q)}$. Set $R = 1-a^{-1}Q$ and define the tiling $\mathcal T'$ with color set $A$ by setting $\frac{1}{1-R} =  \sum_{\balpha \in \ZZ^n} \mathcal T'(\balpha) \bx^{\balpha}$. It follows from Corollary~\ref{cor:substitution-length-p} that $\overline{\mathcal T'}$ is invariant under a linear substitution $\mathcal S \colon W \to W^{\mathcal I(1)^n}$ of length $p$. Since  $\mathcal T = a^{-1} \mathcal T'$, we have $\overline{\mathcal T} =  a^{-1}\overline{ \mathcal T'}$ and since $\mathcal S$ is linear, we infer that $\overline{\mathcal T}$ is invariant under the same linear substitution $\mathcal S$. This proves (i). 

\medskip 
Let us now write $P(\bx) = \sum_{0\leq \bbeta  \leq \bD-1} w(\bbeta) \bx^{\bbeta}$. Hence we have
\begin{align*}
 \sum_{\balpha \in \ZZ^n} \mathcal M(\balpha) \bx^{\balpha} 
 &= \frac P {Q} \\
 & = ( \sum_{0\leq \bbeta  \leq \bD-1} w(\bbeta) \bx^{\bbeta}) ( \sum_{\balpha \in \ZZ^n} \mathcal T(\balpha) \bx^{\balpha})\\
 &= \sum_{\balpha}\big(\sum_{0\leq \bbeta  \leq \bD-1} w(\bbeta)  \mathcal T(\balpha-\bbeta)  \big) \bx^{\balpha},
\end{align*}
so that 	$\mathcal M(\balpha) = 	\sum_{0\leq \bbeta  \leq \bD-1} w(\bbeta)  \mathcal T(\balpha-\bbeta)$ for all $\balpha \in \ZZ^n$.	
Observe that for   any $\bbeta  \in \ZZ^n$ with 	$0\leq \bbeta  \leq \bD-1$, we have $-\bbeta \in \mathcal J(0)^n$. Let $\tau  \colon W \to A $ be the linear form defined by 
$$\tau \colon f \mapsto  \sum_{0\leq \bbeta  \leq \bD-1} w(\bbeta) f(-\bbeta).$$
By the above, we have $\mathcal M(\balpha) = \tau \big( \overline{\mathcal T}(\balpha)\big)$ for all $\balpha \in \ZZ^n$. This proves (ii).

Finally, the assertion (iii) follows from (i) and (ii), in view of Corollary~\ref{cor:tool}. 
\end{proof}

\section{Rational fractions with matrix coefficients}
	
We are now ready to finish the proof of the main result. 

\begin{proof}[Proof of Theorem~\ref{thm:general}]
Given the identification
 $$A[\![ x_1, \dots, x_n] \!] \cong \mathrm{Mat}_{d \times d}(\mathbf F_p[\![ x_1, \dots, x_n] \!])$$
 of the algebra of power series with matrix coefficients with the matrix algebra with power series coefficients, we may  view $P, Q$ and the power series $M =PQ^{-1}$ as $d\times d$-matrices with coefficients in  $\mathbf F_p[\![ x_1, \dots, x_n] \!]$. Using the expression of the inverse matrix $Q^{-1}$ as the adjugate matrix of $Q$ divided by its determinant, we see that for all $i, j \in \{1, \dots, d\}$, the entry $M_{i, j}$ can be expressed as a rational fraction $\frac{P_{i, j}}{Q_0}$, where $P_{i, j}$ and $Q_0 = \det(Q)$ are  polynomials in $\mathbf F_p[ x_1, \dots, x_n] $. Observe   that  the constant term of $Q_0$ is non-zero. Moreover we have  $\deg(Q_0) \leq d\deg(Q)$. 

We  invoke Proposition~\ref{prop:P/1-Q} for each entry $M_{i, j}$ separately. Writing $\mathcal M_{i, j}$ for the $(i, j)$-entry of the tiling $\mathcal M$ defined as $M = \sum_{\balpha} \mathcal M(\balpha) \bx^{\balpha}$,   Proposition~\ref{prop:P/1-Q}(ii) ensures that $\mathcal M_{i, j}(\balpha)  = \tau_{i, j} \big( \overline{\mathcal T}(\balpha)\big)$, for a certain linear form  $ \tau_{i, j} \colon W \to \FF_p$.

We have $M = (M_{i, j})_{1\leq i, j\leq d}$. Viewing the set of linear forms $(\tau_{i, j})_{1\leq i, j\leq d}$ as a linear map $\tau \colon W \to \mathrm{Mat}_{d \times d}(\mathbf F_p)$, we obtain  $\mathcal M = \tau \circ \overline{\mathcal T}$. This proves (i) and (ii). The  assertion (iii) follows, in view of Corollary~\ref{cor:tool}. 
\end{proof}

 \providecommand{\bysame}{\leavevmode\hbox to3em{\hrulefill}\thinspace}
\providecommand{\MR}{\relax\ifhmode\unskip\space\fi MR }

\providecommand{\MRhref}[2]{%
  \href{http://www.ams.org/mathscinet-getitem?mr=#1}{#2}
}
\providecommand{\href}[2]{#2}

  \appendix

\section{A picture gallery}

Figure~\ref{fig2bis} provides four illustrations of Corollary~\ref{cor:2-d} for $d=2$ and $n=2$ over  fields of characteristic $3$ and $5$.

All the remaining figures presented below are illustrations of Theorem~\ref{thm:general} for $d=1$, $n=2$ and $P=1$. Each picture represents a portion of the tiling of the first octant obtained for a specific choice of the characteristic $p$ and the polynomial $Q \in \FF_p[x_1, x_2]$. To simplify the notation we use the symbols $x$ and $y$ for the indeterminates, instead of $x_1, x_2$ as in Theorem~\ref{thm:general}. 

\begin{figure}[h]
\includegraphics[width=6.8cm]{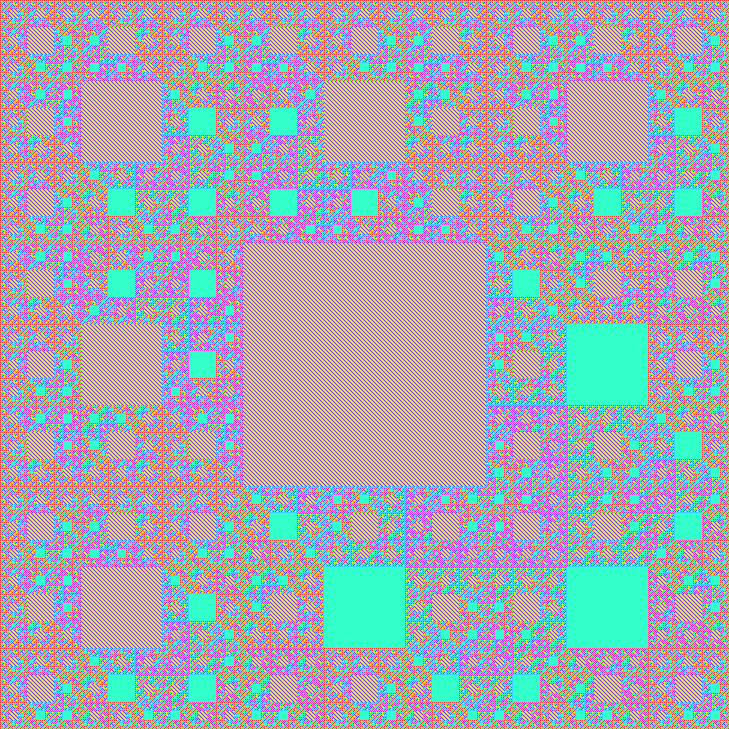} \hspace{.3cm}\includegraphics[width=6.8cm]{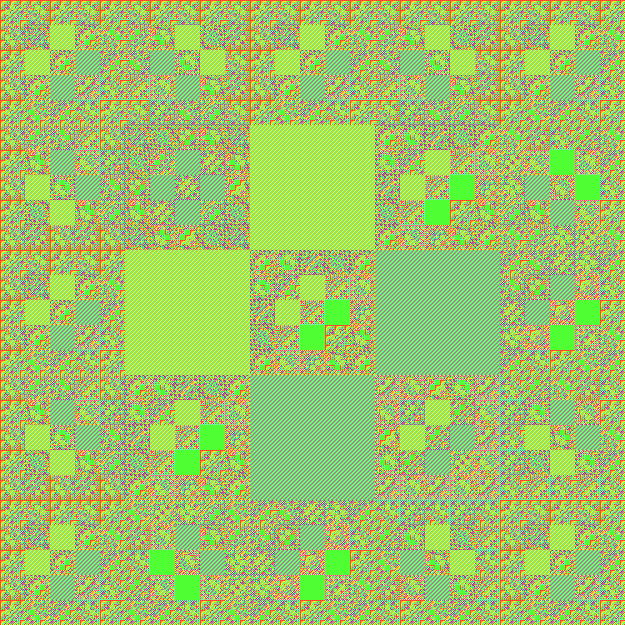}\\
\vspace{.3cm}
\includegraphics[width=6.8cm]{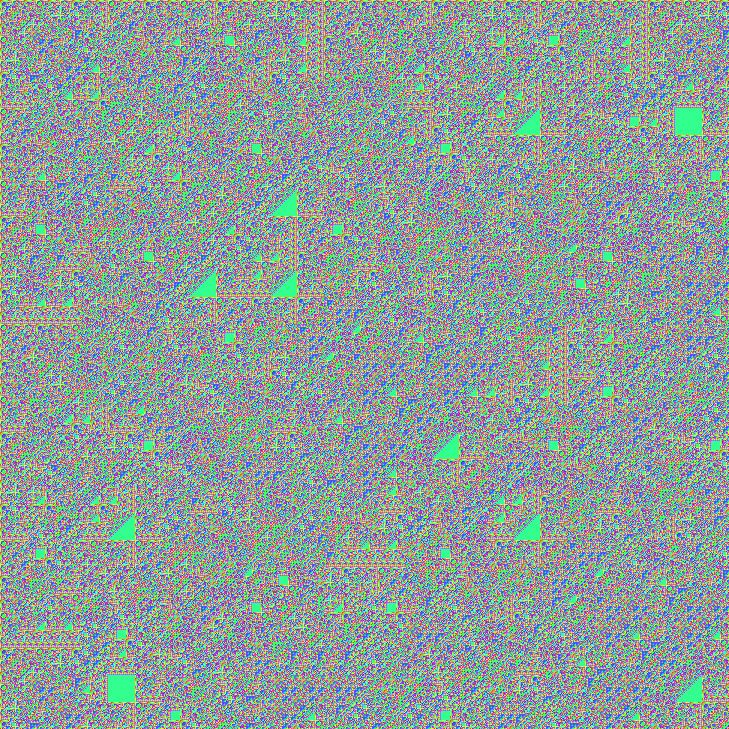} \hspace{.3cm}\includegraphics[width=6.8cm]{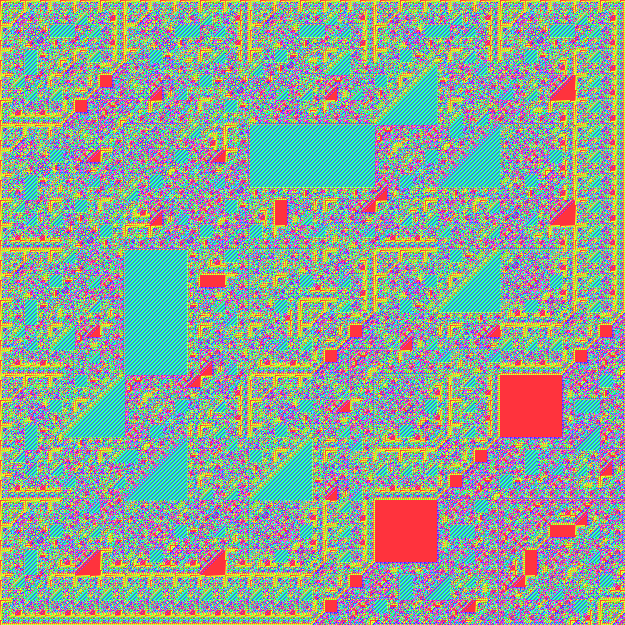}\\
\caption{Other illustrations of Corollary~\ref{cor:2-d} with $d=2$. In the left column $p=3$; in the right column $p=5$. \\
On the top line: $a = b = \left(\begin{array}{cc} 0 & 1 \\ 1 & 0\end{array}\right)$ and $c = \left(\begin{array}{cc} 0 & -1 \\ 1 & -1\end{array}\right)$. \\
On the bottom line: $a = b = \left(\begin{array}{cc} 1 & 1 \\ 1 & 0\end{array}\right)$ and $c = \left(\begin{array}{cc} 0 & 1 \\ 0 & 0\end{array}\right)$.}
\label{fig2bis}
\centering
\end{figure}	

\begin{figure}[h]
\includegraphics[width=6.8cm]{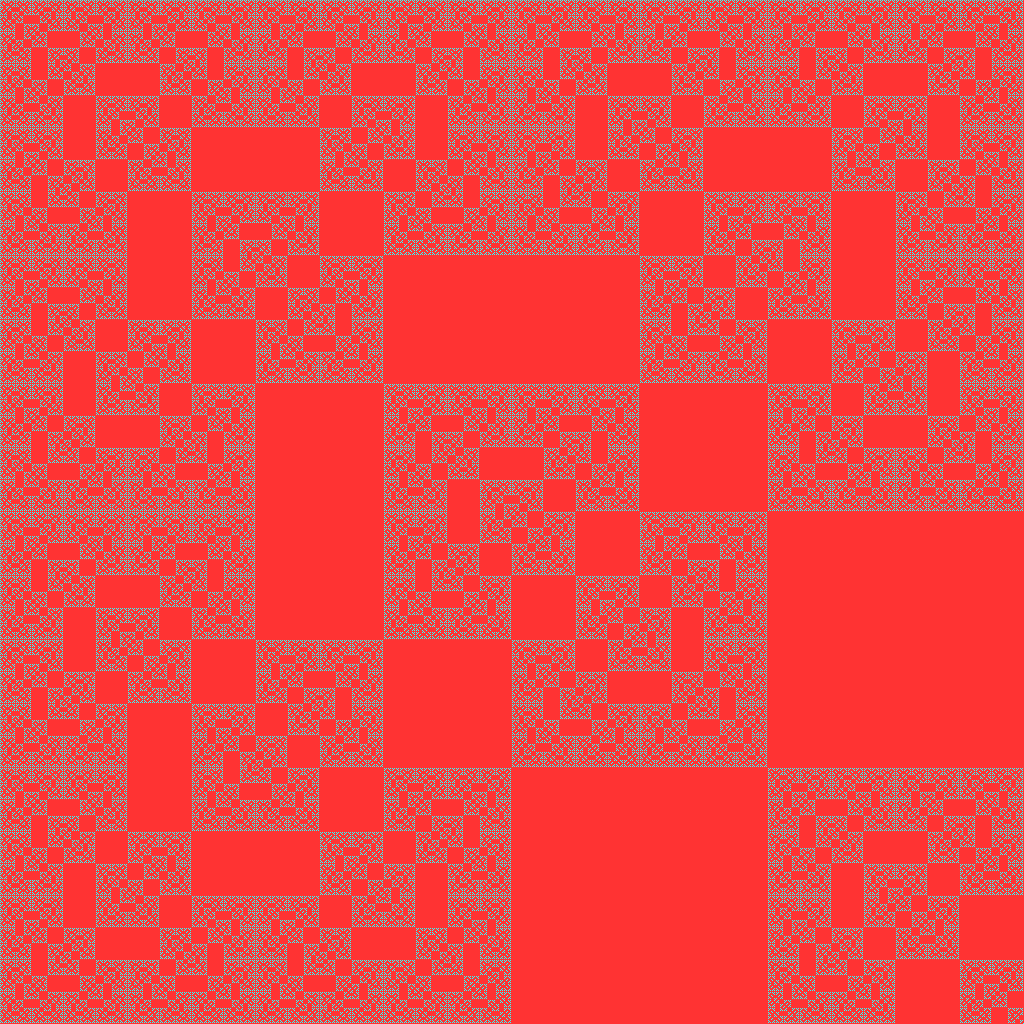} \hspace{.3cm}\includegraphics[width=6.8cm]{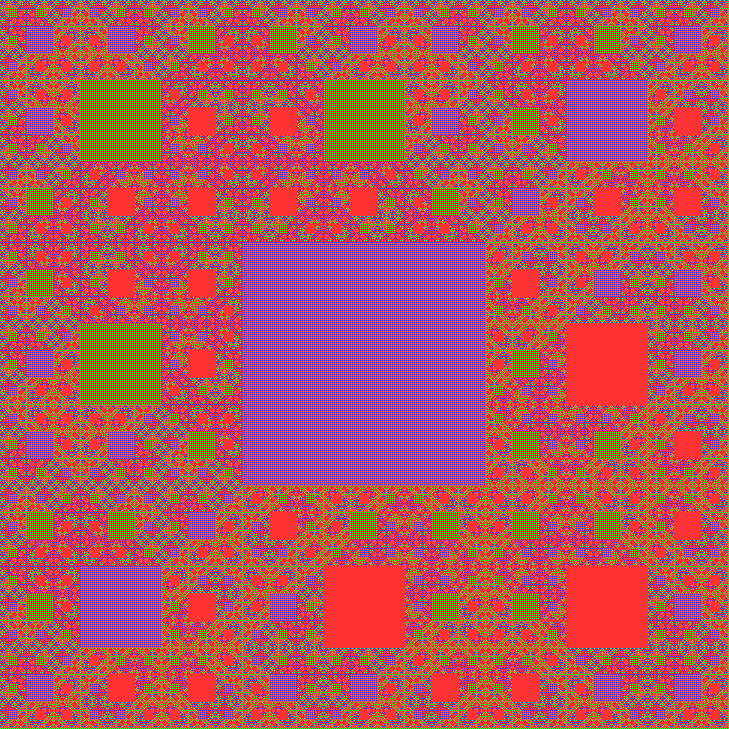}\\
\vspace{.3cm}
\includegraphics[width=6.8cm]{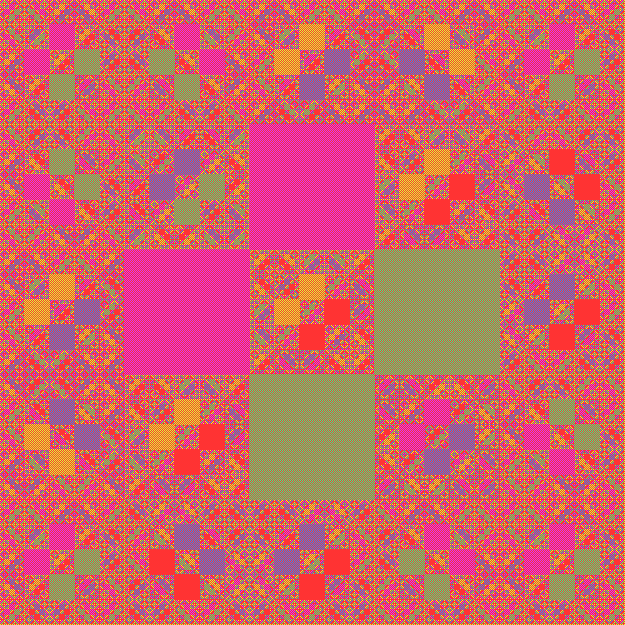} \hspace{.3cm}\includegraphics[width=6.8cm]{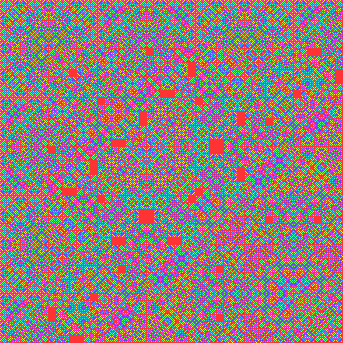}\\

\caption{$Q = x^2y^2+xy+x^2+y^2+1$, $p=2, 3, 5, 7$}
\label{fig3}
\centering
\end{figure}	

\begin{figure}[h]
\includegraphics[width=6.8cm]{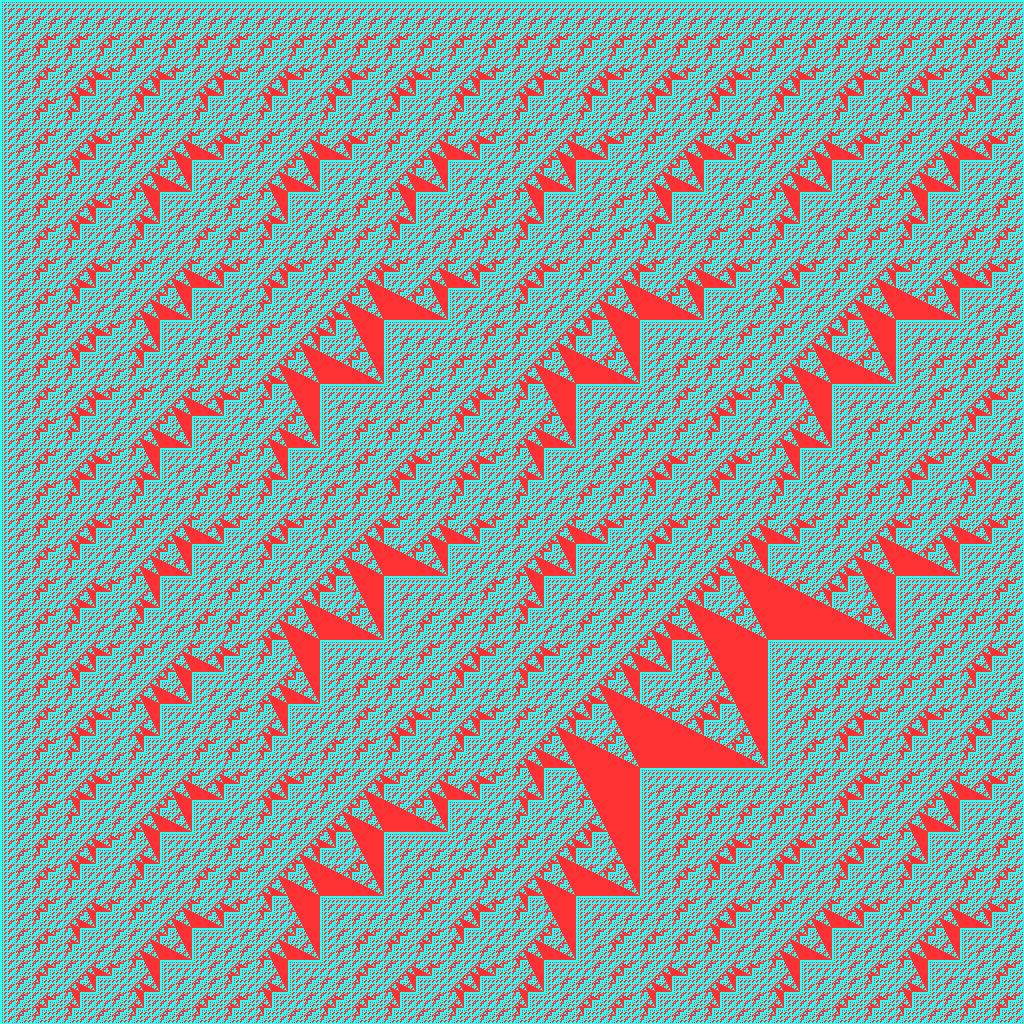} \hspace{.3cm}\includegraphics[width=6.8cm]{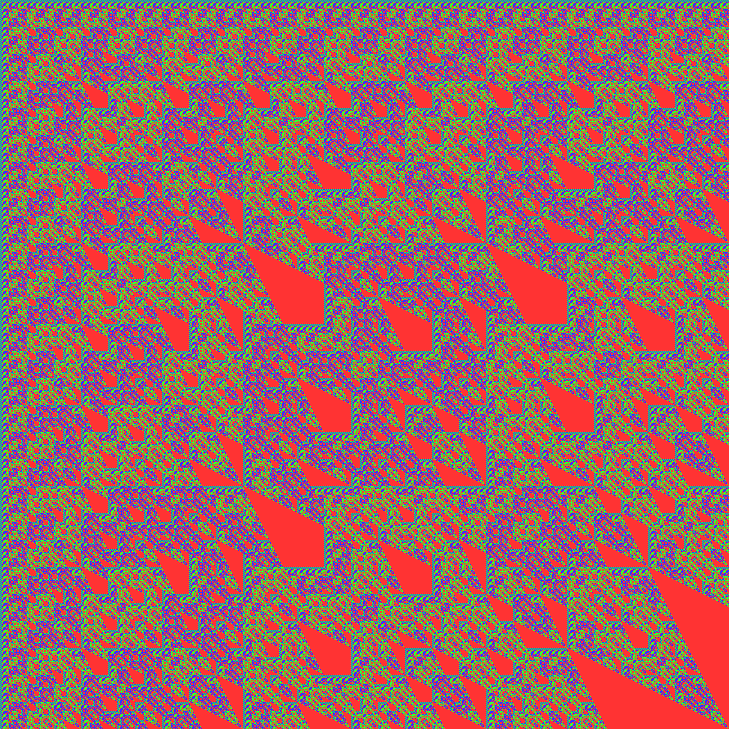}\\
\vspace{.3cm}
\includegraphics[width=6.8cm]{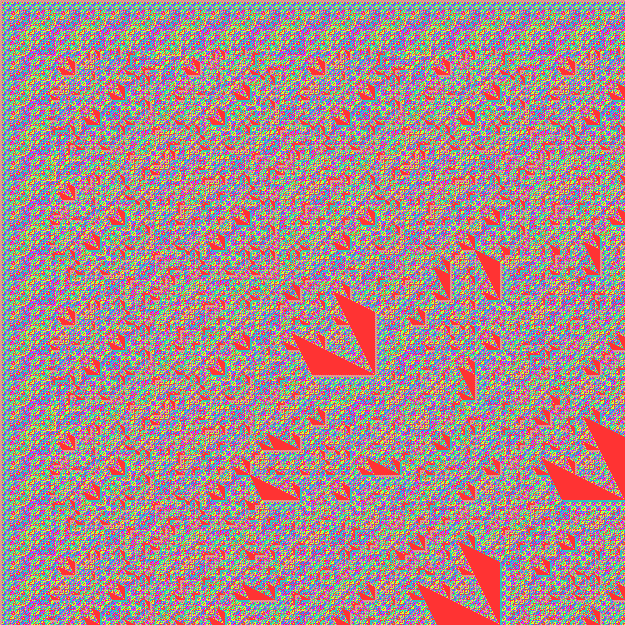} 
\hspace{.3cm}\includegraphics[width=6.8cm]{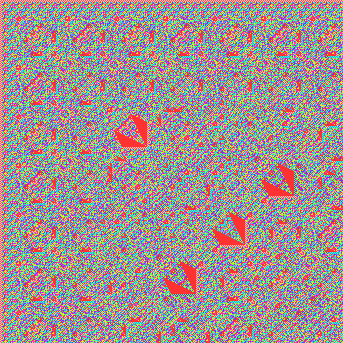}
\\

\caption{$Q = x^2y^2+xy+x+y+1$, $p=2, 3, 5, 7$}
\label{fig4}
\centering
\end{figure}	

\begin{figure}[h]
\includegraphics[width=9.5cm]{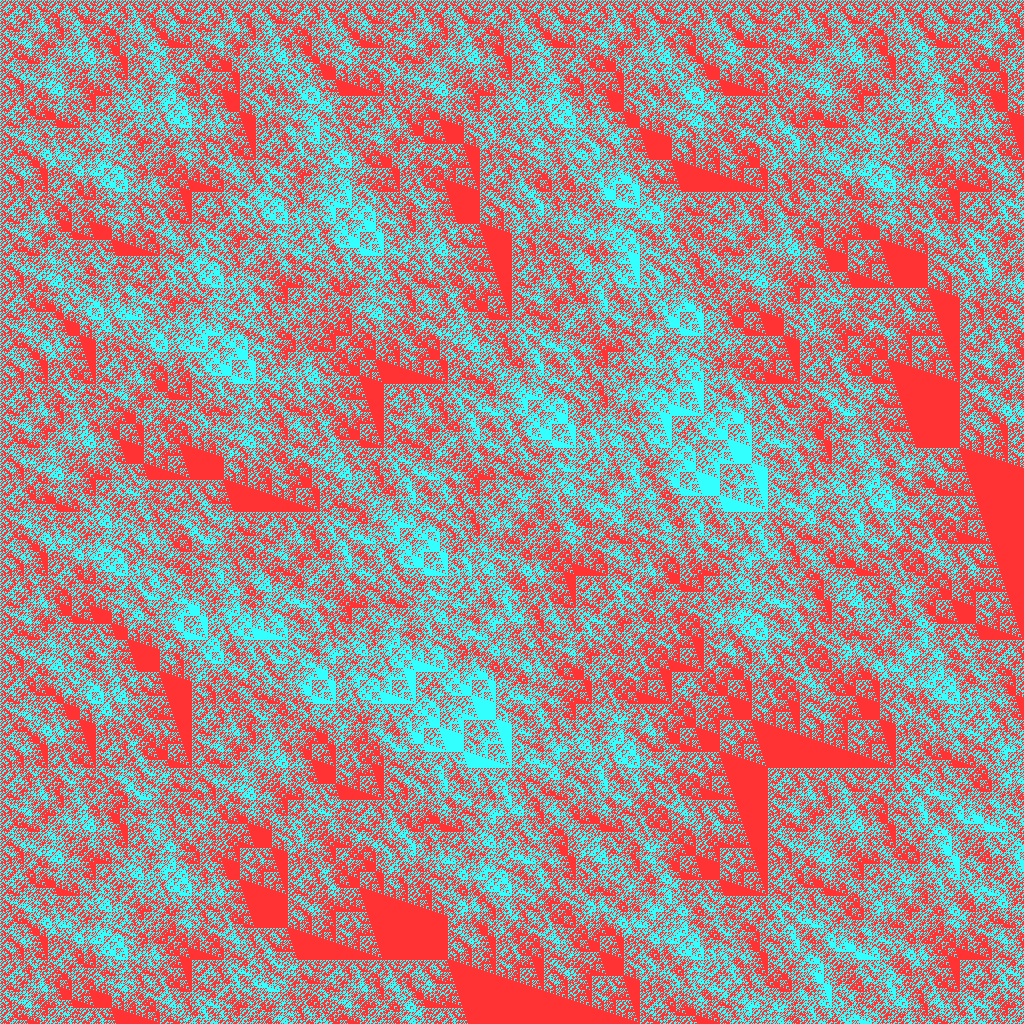}
\caption{ $p=2$, $Q= x^3 y^3 + x^2+ y^2 +x + y +1$}
\label{fig:5}
\end{figure}

\begin{figure}[h]
\includegraphics[width=9.5cm]{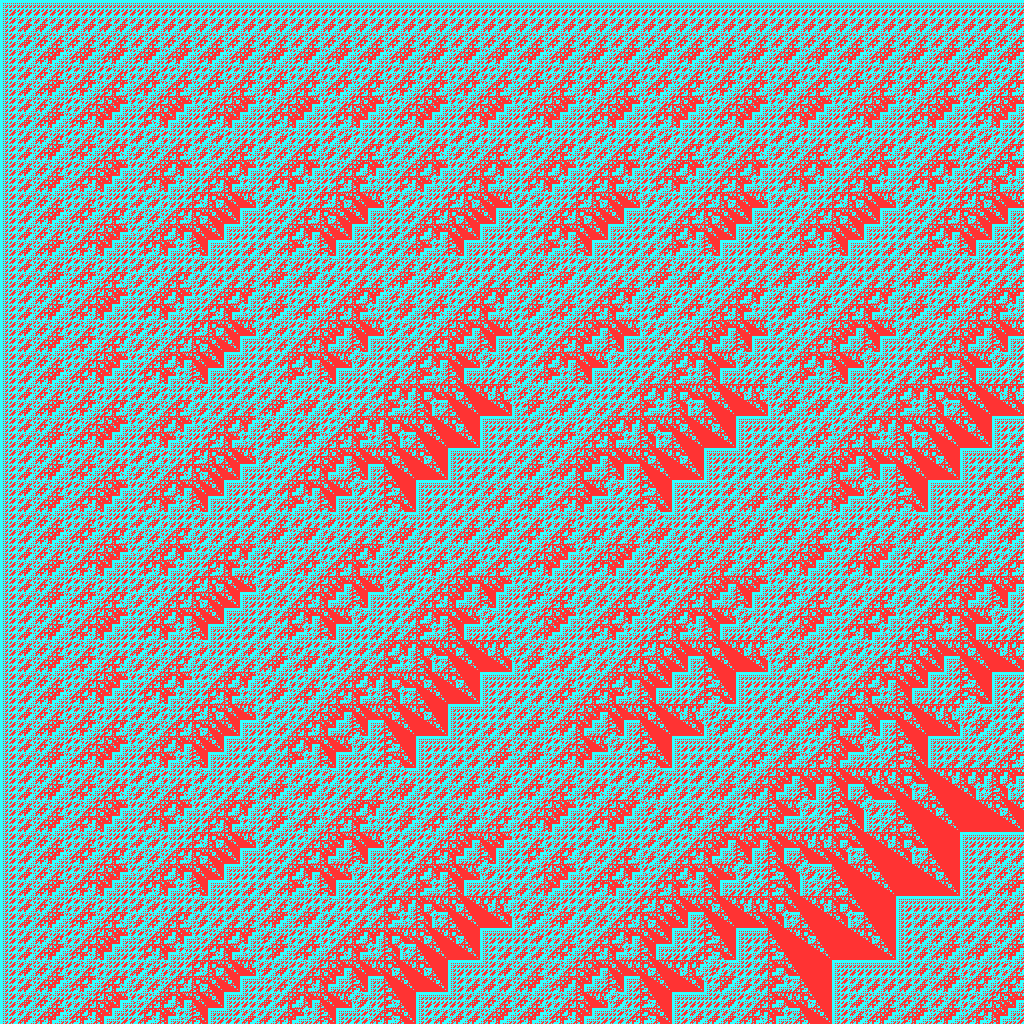}
\caption{ $p=2$, $Q= x^3 y^3 + xy+ x +  y +1$}
\label{fig:6}
\end{figure}	  

\begin{figure}[h]
\includegraphics[width=9.5cm]{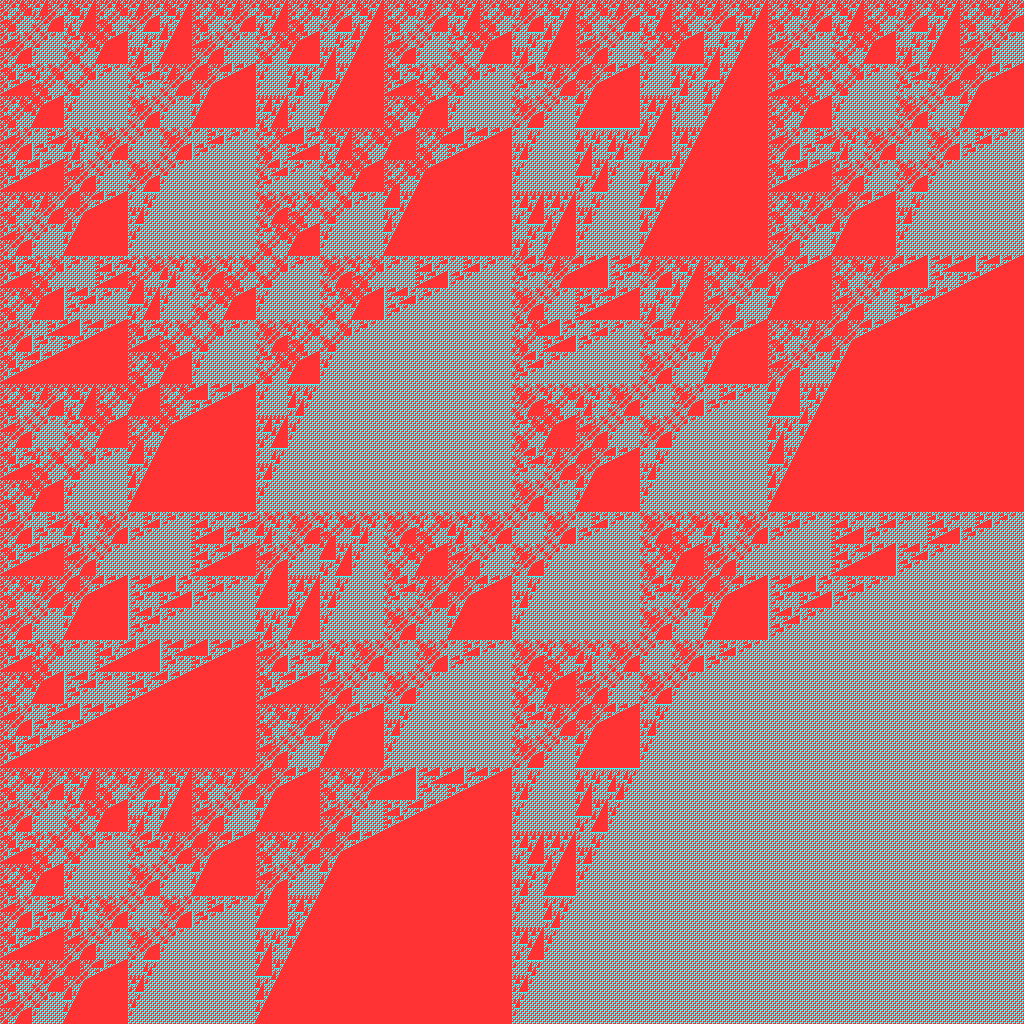}
\caption{ $p=2$, $Q= x^3 +y^3 +x^2y^2+x^2y+xy^2 + 1$}
\label{fig:7}
\end{figure}	

\begin{figure}[h]
\includegraphics[width=9.5cm]{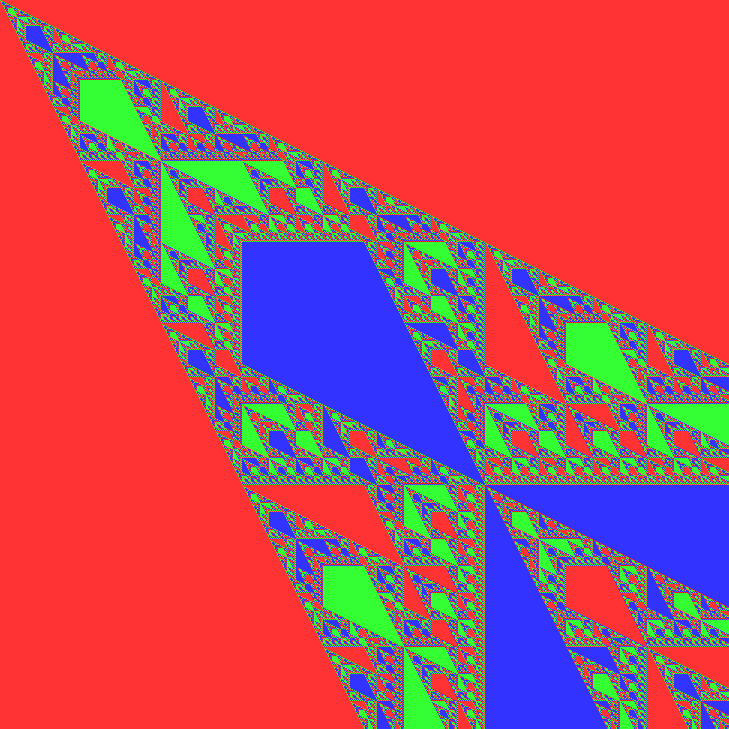}
\caption{ $p=3$, $Q= -x^2y^2 -x^2y -xy^2 -xy+1$ \\
(compare Figure~\ref{fig2} for $p=2$)}
\label{fig:8}
\end{figure}	

\begin{figure}[h]
\includegraphics[width=9.5cm]{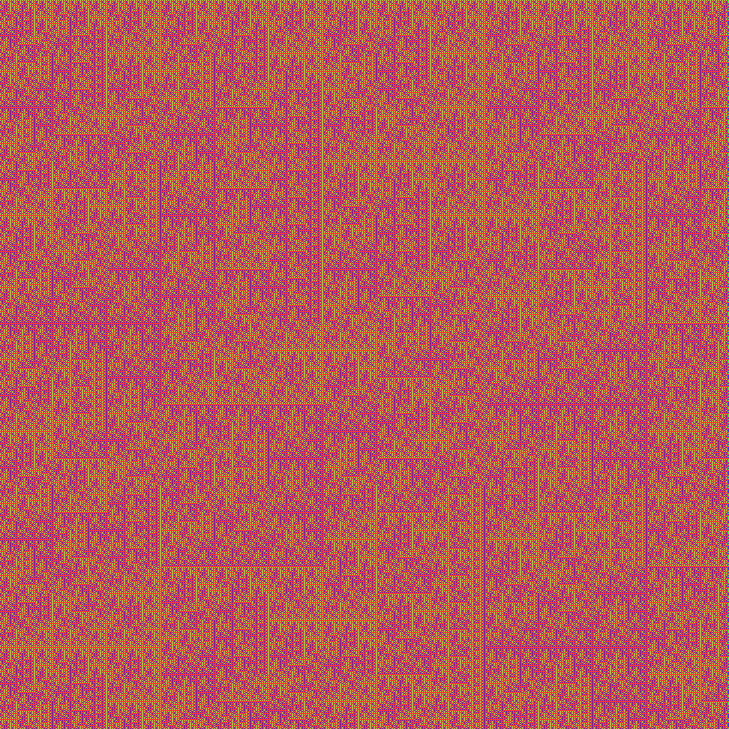}
\caption{ $p=3$, $Q= x^2y^2 -x^2 -y^2 +y+1$}
\label{fig:9}
\end{figure}

\begin{figure}[h]
\includegraphics[width=9.5cm]{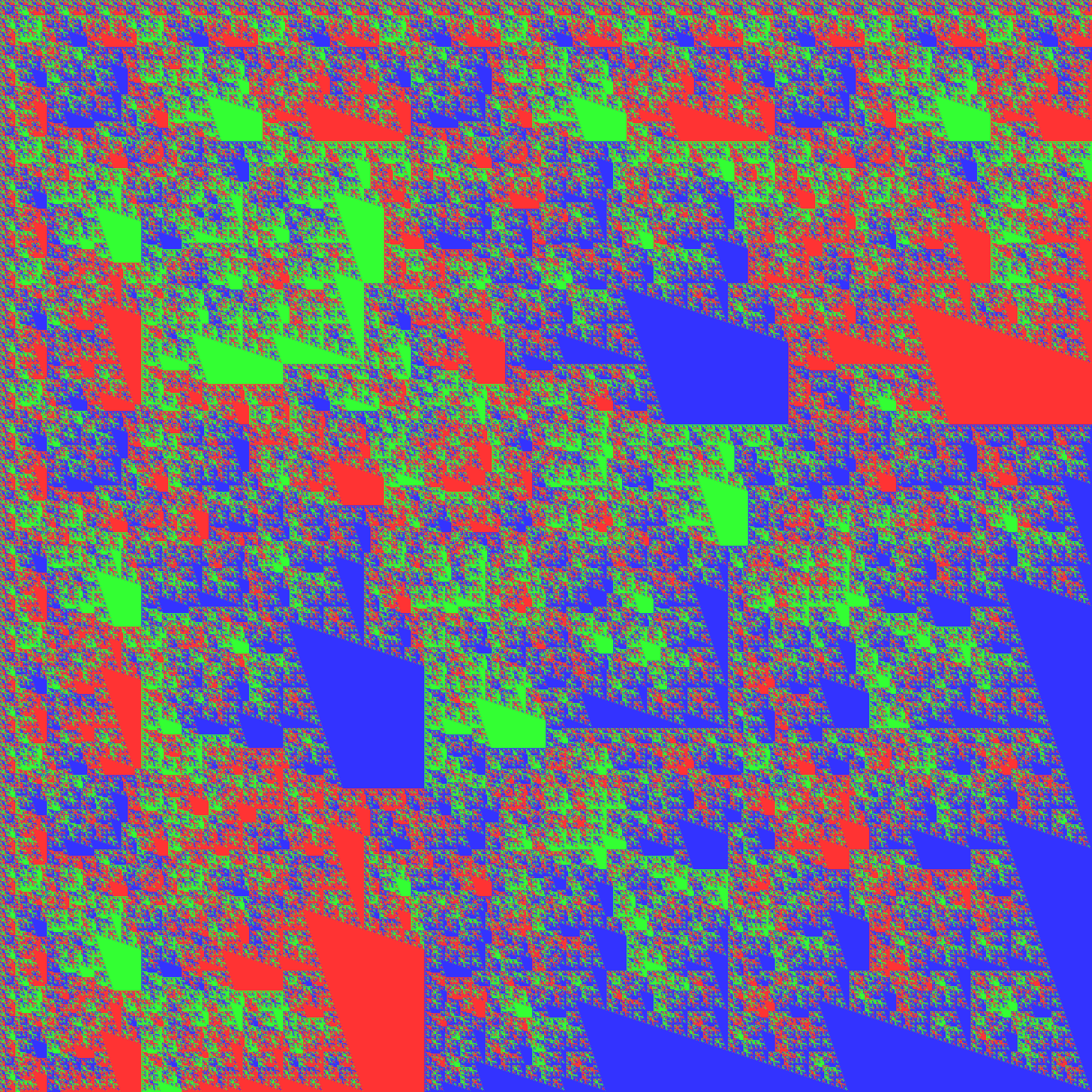}
\caption{ $p=3$, $Q= x^3y^3+x^2+y^2+x+y+1$ \\
(compare Figure~\ref{fig:5} for $p=2$)}
\label{fig:10}
\end{figure}

\end{document}